\documentclass[11pt]{amsart}
\usepackage{amsmath,amsthm,amssymb,enumerate,bbm}
\usepackage{euscript}
\usepackage{color}
\usepackage{dsfont}
\usepackage[left=2cm,right=2cm,top=3.5cm,bottom=3.5cm]{geometry}
\usepackage{hyperref}
\usepackage[capitalise]{cleveref}
\usepackage{mathrsfs}
\usepackage{graphicx}
\usepackage{tkz-euclide}
\usepackage{upref}
\usepackage[utf8]{inputenc}

\allowdisplaybreaks

\hypersetup{linkcolor=blue, colorlinks=true, citecolor = red}
\usepackage{upref}
\usepackage{url}
\usepackage[bottom]{footmisc}

\numberwithin{equation}{section}
\newtheorem{Theorem}{Theorem}[section]
\newtheorem{Lemma}[Theorem]{Lemma}
\newtheorem{Proposition}[Theorem]{Proposition}

\newtheorem{Definition}[Theorem]{Definition}
\newtheorem{Remark}[Theorem]{Remark}
\let\div\relax
\DeclareMathOperator{\div}{\mathrm{div}}

\newcommand{\Id}{\operatorname{Id}}
\newcommand{\dd}{\ \mathrm{d}}
\newcommand{\del}{\partial}

\newcommand{\mf}{\mathcal{F}}

\newcommand{\ms}{\mathcal{S}}

\newcommand{\bu}{\mathbf{u}}
\newcommand{\bS}{\mathbb S}
\newcommand{\bD}{\mathbb D}

\def\hmath$#1${\texorpdfstring{{\rmfamily\textit{#1}}}{#1}}

\title[]{Collision/No-collision results of a solid body with its container in a 3D compressible viscous fluid}

\author{Bum Ja Jin}
\address{Department of Mathematics Education, Mokpo National University, Muan 534-729, South Korea.}
\email{bumjajin@mokpo.ac.kr}

\author{\v S\'arka Ne\v casov\'a}
\address{Institute of Mathematics, Czech Academy of Sciences,
\v Zitn\'a 25, 115 67 Praha 1, Czech Republic.}
\email{matus@math.cas.cz}

\author{Florian Oschmann}
\address{Institute of Mathematics, Czech Academy of Sciences,
\v Zitn\'a 25, 115 67 Praha 1, Czech Republic.}
 \email{oschmann@math.cas.cz}

\author{Arnab Roy}
\address{Basque Center for Applied Mathematics (BCAM), Alameda de Mazarredo 14, 48009 Bilbao, Spain.}
	\address{IKERBASQUE, Basque Foundation for Science, Plaza Euskadi 5, 48009 Bilbao, Bizkaia, Spain.}
\email{aroy@bcamath.org}
\date{\today}

\begin{document}

%\maketitle

\begin{abstract}
We consider a bounded domain $\Omega\subset\mathbb R^3$ and a rigid body $\mathcal{S}(t)\subset\Omega$ moving inside a viscous compressible Newtonian fluid. We exploit the body's roughness to establish that the solid  collides with its container within a finite time. We investigate the case when the boundary of the body is of $C^{1,\alpha}$-regularity and show that collision can happen for some suitable range of $\alpha$. We also discuss some no-collision results for the smooth body case when an additional control is added.
\end{abstract}
\maketitle

\bigskip


%\tableofcontents

%%%%%%%%%%%%%%%%%%%%%%%%%%%%%

%%%%%%%%%%%%%%%%%%%%%%%%%%%%%

\section{Introduction}
We consider the compressible Navier-Stokes equations of a barotropic fluid in a bounded domain $\Omega\subset \mathbb{R}^3$ and a rigid body $\ms(t)$ with center of mass at $\mathbf{G}(t)$ moving inside the fluid, where the fluid domain is $\mf (t)=\Omega\setminus \overline{\ms (t)}$. The equations of motion are given by

\begin{equation}\label{eq1}
\begin{cases}
\partial_t \rho_f+\div(\rho_f \bu_f)=0 & \mbox{ in }\mf(t),\\
\partial_t(\rho_f \bu_f)+\div(\rho_f \bu_f\otimes \bu_f)-\div{\bS(\bu_f)}+\nabla p={ \rho_f {\mathbf f}} & \mbox{ in }\mf(t),\\
\bu_f =\dot{\mathbf G}(t) %+ \omega(t)\times { (x-{\mathbf G}(t))}
& \mbox{ on } {\partial \ms(t)},\\
\bu_f=0 & \mbox{ on } {\partial \Omega},\\
m\ddot{\mathbf G}=- \int_{\partial \ms}\Big({\bS}(\bu_f)- p{\mathbb I}\Big){\mathbf n} \dd S +{ \int_{\ms}\rho_s {\mathbf f} \dd x},\\
%\frac{\mathrm d}{\mathrm d t}(\mathbb J \omega)=-\int_{\partial \ms} (x-{\mathbf G}) \times \Big({\bS}(\bu_f)- p{\mathbb I}\Big){\mathbf n} \dd S
%+ \int_\ms (x-{\mathbf G}) \times \rho_s {\mathbf f} \dd x,\\
\rho_f (0)=\rho_0,\ \rho_f\bu_f(0)=\mathbf q_0,\ \mathbf{G}(0)= \mathbf G_0,\ \dot{\mathbf G}(0)=\mathbf V_0 %,\ \omega(0)=\omega_0
& \mbox{ in }\mf(0).
\end{cases}
\end{equation}
 Here, $\bu_f$ is the fluid's velocity, $\rho_f$ and $\rho_s$ are the fluid's and solid's density, respectively. In the above, $\dot{\mathbf G}(t)$
 %and $\omega(t)$ are
 is the translational
 %and rotational velocities
 velocity of the rigid body. %, respectively. 
 
 The pressure $p$ of the fluid is given by
 \begin{equation}\label{p-law}
 p=p(\rho_f)= (\rho_f)^\gamma\mbox{ for some } \gamma > \frac{3}{2}.
 \end{equation}
 
 The stress tensor satisfies Newton's rheological law
\begin{equation*}
    {\mathbb
S}(\bu)=2\mu \bD(\bu)+\lambda {\mathbb I}\div \bu,
\end{equation*}
 where $\mu>0,\ 2\mu+3\lambda\geq 0$, and $\bD(\bu)=\frac{1}{2}(\nabla \bu+\nabla^T \bu)$ is the symmetric part of the gradient of $\bu$. Further, we assume that $\rho_s>0$ is constant. The mass
 %and moment of inertia
 of the rigid body is given by
\begin{equation*}
    m=\rho_s|\ms(0)|. %,\quad 
%{\mathbb J}(t)=\int_{\ms(t)}\rho_s\Big(|x-{{\mathbf{G}}}|^2{\mathbb I}-(x-{{\mathbf{G}}})\otimes (x-{{\mathbf{G}}})\Big) \dd x.
\end{equation*}
The centre of mass is defined as
\begin{equation*}
    \mathbf{G}(t)=\frac{1}{m}\int_{\ms(t)}\rho_s x \dd x.
\end{equation*}
%Note especially that we have for any $\omega\in \mathbb R^3$
%\begin{align*}
%    \mathbb J \omega\cdot\omega=\int_{\ms(t)} \rho_s \big(|x-\mathbf G|^2 |\omega|^2-|(x-\mathbf G)\cdot\omega|^2\big) \dd x = \int_{\ms(t)} \rho_s |\omega\times (x-\mathbf G)|^2 \dd x \geq 0.
%\end{align*}
%We will also assume that the solid's mass is independent of time, that is, $m=\rho_s |\ms(t)|$ for any $t\geq 0$.

\medskip

Understanding the dynamics of a solid body immersed in a fluid is a very active area of research. During the last years, there are many interesting studies focused on the collision/no collision results of a moving body with the boundary of the domain. However, all these available results (see \cite{VaretHill2012, VHW2015, HES, Hillairet2007, HillairetTakahashi2009} and the references therein) are for the \emph{incompressible} fluid-rigid body interaction system only. The aim of the present paper is therefore to deal with viscous \emph{compressible} fluids.
%Excluded: VaretHillairet2010, HS

The answer to the question whether or not collision occurs mainly depends on the ``physical roughness'' of the system's boundary, which can be further split into two issues: the shape of the falling body, and the boundary conditions imposed on the solid's and container's wall. This can also be translated into the properties of the velocity gradient. A simple argument reveals that the velocity gradient must become singular (unbounded) at the contact point since otherwise the streamlines would be well defined,
in particular, they could never meet each other.

Experimentally, the influence of the particle roughness on the collision problem had been studied in \cite{joseph2003collisional} and the influence of boundary conditions at the fluid–solid interfaces had been investigated in \cite{lauga2005}. Mathematically, the roughness-induced effect on the collision process in two space dimensions is analyzed in \cite{VaretHillairet2010}. They considered a vertical motion of a $C^{1,\alpha}$ rigid body falling over a horizontal flat surface, and proved that collision happens in finite time if and only if $\alpha<1/2$. The authors in \cite{VaretHill2012} investigated the relation between the roughness and the drag force in more detail. They studied the evolution of a three dimensional rough solid falling towards a rough wall, introducing the roughness by special shapes of the solid's and container's boundary, respectively, or by considering Navier slip conditions. The works \cite{Hillairet2007, HillairetTakahashi2009} deal with the free fall of a rigid sphere over a wall with no-slip conditions at the solid's and container's boundary in two and three spatial dimensions, respectively. In this context, the sphere does not collide with the wall. The effect of Navier slip boundary conditions was analyzed in \cite{VHW2015}. They established that under certain assumptions on the slip coefficients the rigid ball touches the boundary of the domain in finite time. In addition, the authors in \cite{NP} showed that collision appears under the prescription of the motion of a given ball in the case of slip boundary conditions.

Let us also mention the paper \cite{STA1}, where the author %constructed colliding solutions by adding an $H^{-1}$ source term.
%Starovoitov \cite{STA1}
constructed a collision example of a rigid body with the boundary of the physical 2D domain
resulting from the action of a very singular driving force. On the
other hand, the same author \cite{Staro2004} showed that collisions, if
any, must occur with zero relative {\it translational} velocity
as soon as the boundaries of the rigid objects are smooth and the
gradient of the underlying velocity field is square-integrable - a
hypothesis satisfied by any Newtonian fluid flow of finite energy. The collision of a rigid body with the wall, moving in compressible non-Newtonian fluids or heat-conducting Newtonian fluids, has been  discussed in \cite{nevcasova2024collision}. The impact of fluid compressibility on the contact problem has been examined in \cite{kytomaa1992collision, sundararajakumar1996non}.
%The possibility or impossibility of collisions in viscous fluids is
%related to the properties of the velocity gradient. A simple argument
%reveals that the velocity gradient must become singular (unbounded) at the
%contact point since otherwise the streamlines would be well defined,
%in particular, they could never meet each other.

%Let us also mention the particular case where the authors showed under the prescription of the motion of a given ball the collision appears in the case of slip boundary conditions, see \cite{NP}.

Concerning non-Newtonian incompressible fluids, the existence of global-in-time solutions of the motion of several rigid bodies was proven in \cite{FHN}, where, in accordance with \cite{Staro2004},
collisions of two or more rigid objects do not appear in finite
time unless they were present initially.

In this paper we want to discuss the collision phenomenon of a moving rigid body with the boundary of its container when the container is filled with a viscous \emph{compressible} fluid. Precisely, we want to construct an example in which the rigid body collides with the boundary of the domain filled by the compressible Newtonian fluid. We want to exploit the roughness of the body to achieve this collision result, thus considering the case when the boundary of the body is of $C^{1,\alpha}$-regularity and searching for a suitable range for the exponent $\alpha$ such that the collision can happen. We also present an instance where a no-collision result holds for a smooth rigid body under an additional control force.

%%%%%%%%%%%%%%%%%%%%%%%%%%%%%

%%%%%%%%%%%%%%%%%%%%%%%%%%%%%

\subsection{Weak formulation and Main result}

Let us introduce the notion of weak solutions to problem \eqref{eq1}.
First, we extend $\rho$ and $\bu$ to the whole of $\mathbb{R}^3$ via
\begin{align}\label{extended}
\rho = \begin{cases}
\rho_f & \mbox{ in }\mf(t), \\
\rho_s & \mbox{ in }\ms(t), \\
0 & \mbox{ in }\mathbb{R}^3 \setminus \Omega,
\end{cases}
\quad \bu=\begin{cases}
\bu_f & \mbox{ in }\mf (t), \\
\dot{\mathbf G}(t) %+\omega(t) \times (x-\mathbf{G}(t))
& \mbox{ in }\ms (t), \\
0 & \mbox{ in }\mathbb{R}^3 \setminus \Omega.
\end{cases}
\end{align}
Then we consider the following notion of weak solutions:
\begin{Definition}\label{weaksol:def}
A triplet $(\rho, \bu, \mathbf{G})$ is a \emph{renormalized finite energy weak solution} to \eqref{eq1} if:
\begin{itemize}
\item The solution belongs to the regularity class
\begin{gather*}
\rho \geq 0, \quad
\rho \in L^{\infty}(0,T; L^{\gamma}(\Omega)) \cap C([0,T];L^1(\Omega)),\quad
 \bu\in L^2(0,T; W_0^{1,2}(\Omega)),
\\
\mathbf G\in W^{1,\infty}(0,T),\quad \bu=\dot{\mathbf G}(t) %+\omega(t) \times (x-\mathbf{G}(t))
\mbox{ in }\ms(t);
\end{gather*}

\item The weak form of the continuity and renormalized continuity equation hold:
\begin{gather*}
\int_0^T \int_{\mathbb{R}^3} \left[\rho \del_t \phi+(\rho \bu)\cdot \nabla \phi\right] \dd x \dd t =0, \label{weak density}\\
\int_0^T \int_{\mathbb{R}^3} \left[b(\rho) \partial_t \phi+(b(\rho) \bu)\cdot \nabla \phi+\left(b(\rho)-b'(\rho)\rho\right)\div \bu\,\phi\right] \dd x \dd t =0, \label{renormalized}
\end{gather*}
for any $\phi \in C_c^{\infty}((0,T)\times \mathbb{R}^3)$ and for any $b \in C^1(\mathbb{R})$ such that $b'(z)=0$ for $z$ large enough;

\item The weak form of the momentum equation holds for a.e.~$t\in [0,T]$:
\begin{multline}\label{momentum_weak}
\int_0^t \int_\Omega \left[(\rho\bu)\cdot \partial_t \phi+ (\rho\bu \otimes \bu) : \bD(\phi) + p(\rho) \div \phi - \bS(\bu) : \bD(\phi) + \rho\mathbf{f}\cdot\phi\right] \dd x \dd t\\
= \int_\Omega [\rho(t)\bu(t)\cdot \phi(t) - \mathbf q_0\cdot \phi(0)]\dd x
 \end{multline}
for any $\phi \in C_c^\infty([0,T)\times \Omega)$ with $\phi(t,y)=\ell_\phi(t)$ %+ \omega_\phi(t)\times (y-\mathbf G(t))$
in a neighborhood of $\ms(t)$; 

\item The following energy inequality holds for a.e.~$t \in [0,T]$:
\begin{align}\label{StdEI}
\begin{split}
\int_\Omega\left( \frac12 \rho(t,x)|\bu(t,x)|^2+\frac{\rho^\gamma(t,x)}{\gamma-1} \right) \dd x + \int_0^t\int_\Omega\left(2\mu |\bD(\bu)|^2+\lambda|\div\bu|^2\right)(\tau, x) \dd x \dd \tau\\
\leq \int_\Omega \left( \frac12 \frac{|\mathbf q_0(x)|^2}{\rho_0(x)} + \frac{\rho_0^\gamma(x)}{\gamma-1}\right)\dd x + \int_0^t\int_\Omega \rho\mathbf{f}\cdot\bu \dd x,
\end{split}
\end{align}
where we extended $\rho,\mathbf u$ as in \eqref{extended}, and $\mathbf q_0=\rho(0) \bu(0)$.
\end{itemize}
\end{Definition}

Let us remark that the regularity classes of $\rho$ and $\bu$ immediately imply
\begin{align}\label{bdRu}
    \rho\bu \in L^\infty(0,T;L^\frac{2\gamma}{\gamma+1}(\Omega)).
\end{align}
Before stating precisely the well-posedness result of system \eqref{eq1}, we want to mention previous mathematical results related to the existence theory of solutions to this system. Regarding the evolution of a system of a rigid body in a compressible fluid with Dirichlet boundary conditions,
the existence of a weak solution up to collision is proved in \cite{DEES2}. In \cite{Feireisl2004} this result was generalized to allow also for collisions. Existence of weak solutions for Navier slip boundary conditions at the interface as well as the boundary of the domain has been established in \cite{NRRS22}. The existence of strong solutions was studied in \cite{HMTT, roy2019stabilization}.
%Excluded: BG, HMTT
\footnote{The mathematical analysis of systems describing the motion of a rigid body in a viscous \textit{incompressible} fluid is nowadays well developed.  The proof of existence of weak solutions until a first collision can be found in several papers, see e.g.~\cite{DEES1,GLSE}. Later, the possibility of collisions in the case of weak solutions was included, see e.g.~\cite{F3, SST}.}\\
%Excluded: CST, HOST, SER3

We now state the existence result for system \eqref{eq1} which can be established by following \cite[Theorem~4.1]{Feireisl2004}:
\begin{Theorem}\label{exist:upto_collision}
  Let $\Omega$ and $\ms_0$ be two bounded domains of $\mathbb{R}^3$.
  Let $p$ be defined through \eqref{p-law}, and $\mathbf{f}=\nabla F$ with $F = -gx_3$, where $g>0$ is the acceleration due to gravity. Assume that the initial data satisfy 
  \begin{align} \label{init}
  \rho_0 \in L^\gamma(\Omega),\quad \rho_0 \geq 0 &\mbox{ a.e.~in }\Omega,\\
  \label{init1}
  \mathbf q_0 \in L^\frac{2\gamma}{\gamma+1}(\Omega), \quad \mathbf q_0\mathds{1}_{\{\rho_0=0\}}=0 &\mbox{ a.e.~in }\Omega,\quad \dfrac{|\mathbf q_0|^2}{\rho_0}\mathds{1}_{\{\rho_{0}>0\}}\in L^1(\Omega),\\ \label{init2}
  \bu_0= \mathbf V_0 %+ \omega_0 \times (x-\mathbf G_0)
  &\mbox{ on }\ms_0 \mbox{ with } \mathbf V_0, %\omega_0,
  \mathbf G_0 \in \mathbb{R}^3.
  \end{align}
Then the system \eqref{eq1} admits a weak solution in the sense of Definition \ref{weaksol:def}.
  \end{Theorem}
In this paper, we study whether the solid body can collide with the boundary $\del \Omega$ in finite time or not, that is, whether the maximal existence time $T_*$ of a weak solution is finite or not. We consider a $C^{1,\alpha}$ solid moving vertically over a flat horizontal surface under the influence of gravity. More precisely, we make the following assumptions (see Figure~\ref{fig1} for the main notations):
\begin{enumerate}
\item The source term is provided by the gravitational force ${\mathbf f}=-g{\mathbf e}_3$ and $g>0$. \label{a1}
\item The solid moves along the vertical axis $\{x_1=x_2=0\}$. \label{a2}
\item The only possible collision point is at $x=0$, and the solid's motion is a vertical translation.\label{a3}
\item Near $r=0$, $\partial \Omega$ is flat and horizontal, where $r=\sqrt{x_1^2+x_2^2}$. \label{a4}
\item Near $r=0$, the lower part of $\partial \ms(t)$ is given by \label{a5}
\begin{align*}
    x_3={ h(t)+r^{1+\alpha}},\ r\leq 2r_0\mbox{ for some small enough } r_0 \in (0,1).
\end{align*}
%\[ x_3={ h(t)+r^{1+\alpha}},\ r\leq 2r_0\mbox{ for some small enough } r_0>0.\] \label{a5}
\item The collision just happens near the flat boundary of $\Omega$: \label{a6}
\[
\inf_{t>0}\mbox{\rm dist} \left(\ms(t),\partial\Omega \setminus [-r_0, r_0]^2\times\{0\}\right)\geq d_0>0.
\]

\end{enumerate} 

Let us also assume that the position of the solid is characterized by its height $h(t)$, in the sense that ${\mathbf G}(t)={\mathbf G}(0)+(h(t)-h(0)){\mathbf e}_3$, and $\ms(t)=\ms(0)+(h(t)-h(0)){\mathbf e}_3$. We emphasize that this means the body will not rotate during the free fall.

\begin{figure}\label{fig1}
\centering
\begin{tikzpicture}[scale=.7]
%axes and \Omega
\draw[->] (-6,0) -- (6,0);
\draw[->] (0,0) -- (0,7);
\node at (6,0) [anchor=north] {$r$};
\node at (0,7) [anchor=east] {$x_3$};
\draw (-5,0) rectangle (5,6);
\node at (-6,4) [anchor=west] {$\del \Omega$};
%body S, fluid F
\draw[black, thick] plot [smooth cycle] coordinates {(0,1) (1,1.5) (2,3) (3,5) (0,4) (-3,5) (-2,3) (-1,1.5)};
\draw[<->] (0,0) -- (0,1);
\node at (-1,3) {$\ms$};
\node at (-3,1) {$\mf$};
%notations
\node at (0,.5) [anchor=west] {$h$};
\node at (2,2) [anchor=west] {$x_3=h+r^{1+\alpha}$};
\draw[->] (2,2) -- (.75,2) -- (.75,1.4);
\draw[dashed] (-1,0) -- (-1,1.5);
\draw[dashed] (1,0) -- (1,1.5);
\node at (1,0) [anchor=north] {$2r_0$};
\draw[dotted] (-.5,0) -- (-.5,1.1);
\draw[dotted] (.5,0) -- (.5,1.1);
\node at (-.5,0) [anchor=north] {$-r_0$};
\end{tikzpicture}
\caption{The body $\ms$ and fluid $\mf$ in the container $\Omega$}
\end{figure}

Our main result regarding collision now reads as follows:

\begin{Theorem} \label{theo1}

Let $0<\alpha\leq 1$ and $\Omega,\ \ms\subset\mathbb R^3$ be bounded domains of class $C^{1,\alpha}$.
Let $(\rho, \bu, {\mathbf G})$ be a renormalized finite energy weak solution of the compressible Navier-Stokes equations \eqref{eq1} in the sense of Definition \ref{weaksol:def} satisfying the assumptions \eqref{a1}--\eqref{a6}. If the solid's mass is large enough, and its initial vertical
% and rotational
velocity is small enough, then the solid touches $\partial \Omega$ in finite time provided
\begin{align*}
   \gamma>3 \quad \text{and} \quad \alpha<\min \bigg\{ \frac13, \frac{3(\gamma-3)}{4\gamma+3} \bigg\}. 
\end{align*}
\end{Theorem}

\begin{Remark}
The terms ``large enough'' and ``small enough'' should be interpreted in such a way that inequality \eqref{finalIneq} is satisfied.
\end{Remark}

\begin{Remark}
Let us mention a few facts about the above constraint. First, the two expressions inside the minimum stem, as one shall expect, from estimating the diffusive and convective parts, respectively. Second, the first fraction inside the minimum wins precisely if $\gamma \geq 6$. This seems to be optimal in the sense that $\alpha=\frac13$ is the critical value for the incompressible case, which would (loosely speaking) correspond to $\gamma=\infty$ (see \cite{VaretHill2012} for details).
\end{Remark}

\begin{Remark}
In contrast to \cite{VaretHillairet2010, VaretHill2012}, we do not know whether or not collision can occur above the given threshold for $\alpha$. We will comment this issue in more detail at the end of Section~\ref{sec3}.
\end{Remark}

\begin{Remark}
We remark that the above result also holds for a rotational setting, that is, for the velocity on $\ms$ being of the form $\bu = \dot{\mathbf G}(t) + \omega(t) \times (x-\mathbf G(t) )$, together with obvious modifications in the system \eqref{eq1}, provided the initial rotational velocity $\omega(0)$ is small enough. For this case, however, we need to \emph{assume} that the body moves just vertically without rotating during the free fall.
\end{Remark}

\noindent
\textbf{Organization of the paper.} In Section \ref{sec2}, we give \textit{a priori} bounds on the velocity and the density needed for the sequel. Finally, Section \ref{sec3} is devoted to the construction of an appropriate test function and to the proof of Theorem \ref{theo1}. In Section \ref{sec4}, we will discuss some particular scenarios of no-collision between a smooth body and the boundary.

%%%%%%%%%%%%%%%%%%%%%%%%%%%%%

%%%%%%%%%%%%%%%%%%%%%%%%%%%%%

\section{A priori bounds and energy estimates}\label{sec2}
From the energy inequality \eqref{StdEI} we obtain \textit{a priori} bounds on the density and velocity.
\begin{Proposition} \label{prop:EI}

Let $\gamma>1$ and $({\rho}, {\bu}, \mathbf{G})$ be a weak solution to \eqref{eq1} on $(0,T)\times\Omega$ with external force $\mathbf{f}=\nabla F \in L^\frac{2\gamma}{\gamma-1}(\Omega)$, where $F\in L^\frac{\gamma}{\gamma-1}(\Omega)$ is a time-independent potential. Then,
%COMMENT: Factor 2(\gamma-1) on LHS is correct due to Young's inequality in force term. If you want to remove it, you have factor 2 on RHS of inequality.
\begin{align}
\label{eq2}
\begin{split}
&\sup_{t\in (0,T)} \int_{\mf(t)} \bigg(\frac{1}{2}{\rho}|\bu|^2+\frac{\rho^\gamma}{2(\gamma-1)}\bigg) \dd x + \sup_{t\in (0,T)} \int_{\ms(t)} \frac12 \rho |\bu|^2 \dd x \\
&\quad + \int_0^T\int_{\mf(t)} \big( 2\mu|\bD(\bu)|^2+\lambda|\div\bu|^2 \big) \dd x \dd t\\
&\leq \int_{\mf(0)}\bigg(\frac{1}{2}\frac{|\mathbf q_0|^2}{{\rho}_0} + \frac{\rho_0^\gamma}{\gamma-1}\bigg) \dd x+\frac{m}{2} |\mathbf V_0|^2 %+ \frac12 \mathbb J_0 \omega_0\cdot\omega_0
+ C(\gamma)\, \|F\|_{L^\frac{\gamma}{\gamma-1}(\Omega)}^\frac{\gamma}{\gamma-1},
\end{split}
\end{align}
where $C(\gamma)=2^{1/(\gamma-1)}(2-2/\gamma)^{\gamma/(\gamma-1)}\leq 3$.
\end{Proposition}

\begin{Remark}
We remark that $\mathbf f=\nabla F\in L^\frac{2\gamma}{\gamma-1}(\Omega)$, together with the bound \eqref{bdRu}, imply that the force term $\int_\Omega \rho\mathbf f\cdot\bu \dd x$ is well-defined.
\end{Remark}
\begin{proof}[Proof of Proposition \ref{prop:EI}]
By the energy inequality \eqref{StdEI}, we have
\begin{multline*}
\int_\Omega\Big(\frac{1}{2}{\rho}|\bu|^2+ \frac{\rho^\gamma}{\gamma-1}\Big) \dd x
+\int_0^T\int_{\mf(t)} \big( 2\mu|\bD(\bu)|^2+\lambda|\div\bu|^2 \big) \dd x \dd t\\
\leq \int_\Omega\bigg( \frac{1}{2}\frac{|\mathbf q_0|^2}{{\rho}_0} + \frac{\rho_0^\gamma}{\gamma-1}\bigg) \dd x + \int_0^T\int_\Omega \rho\mathbf{f}\cdot\bu \dd x \dd t.
\end{multline*}
We split the pressure term in the first integral to get
\begin{align*}
    \int_\Omega \frac{\rho^\gamma}{\gamma-1} \dd x - \int_\mf \frac{\rho^\gamma}{\gamma-1} \dd x = \int_\ms \frac{\rho^\gamma}{\gamma-1} \dd x = \frac{\rho_s^\gamma}{\gamma-1} |\ms| = \frac{\rho_s^{\gamma-1}}{\gamma-1} m,
\end{align*}
where we used that $\rho|_\ms=\rho_s$ is constant, and the mass of $\ms$ is independent of time. Obviously, the same holds if we replace $\rho$ by $\rho_0$. Thus, the energy inequality turns into
\begin{align}\label{eq2-1}
\begin{split}
    \int_{\mf(t)} \Big(\frac{1}{2}{\rho}|\bu|^2+ \frac{\rho^\gamma}{\gamma-1}\Big) \dd x + \int_{\ms(t)} \frac12 \rho |\bu|^2 \dd x + \int_0^T\int_{\mf(t)} \big( 2\mu|\bD(\bu)|^2+\lambda|\div\bu|^2 \big) \dd x \dd t\\
\leq \int_{\mf(0)}  \bigg( \frac{1}{2}\frac{|\mathbf q_0|^2}{\rho_0} + \frac{\rho_0^\gamma}{\gamma-1}\bigg) \dd x + \int_{\ms(0)} \frac{1}{2}\frac{|\mathbf q_0|^2}{\rho_0} \dd x + \int_0^T\int_\Omega \rho\mathbf{f}\cdot\bu \dd x \dd t.
\end{split}
\end{align}

Next we note that the continuity equation is satisfied in the whole of $\Omega$ due to the specific form of $\bu$ on $\ms$. Since $\mathbf f=\nabla F$ and $F\in L^\frac{\gamma}{\gamma-1}(\Omega)$ is time-independent, we obtain
\begin{align*}
    \int_0^T\int_\Omega \rho\bu\cdot\mathbf f \dd x \dd t &=\int_0^T\int_\Omega \rho\bu\cdot\nabla F \dd x \dd t=-\int_0^T \int_\Omega \div(\rho\bu) F \dd x \dd t\\
    &=\int_0^T \int_\Omega \del_t \rho F \dd x \dd t=\int_\Omega \rho(T) F-\rho_0 F \dd x = \int_{\mf(T)} \rho(T)F \dd x - \int_{\mf(0)} \rho_0 F \dd x,
\end{align*}
where in the last equality we used again that $\rho|_\ms=\rho_s$ is constant. Thus, we estimate the force term as
\begin{align*}
     \bigg| \int_0^T\int_\Omega \rho\bu\cdot\mathbf f \dd x \dd t \bigg| &\leq 2 \sup_{t\in (0,T)} \|\rho F\|_{L^1(\mf(t))}\leq 2\|\rho\|_{L^\infty(0,T;L^\gamma(\mf(\cdot))}\|F\|_{L^\frac{\gamma}{\gamma-1}(\Omega)}\\
     &\leq \frac{1}{2(\gamma-1)}\|\rho\|_{L^\infty(0,T;L^\gamma(\mf(\cdot)))}^\gamma + C(\gamma)\|F\|_{L^\frac{\gamma}{\gamma-1}(\Omega)}^\frac{\gamma}{\gamma-1}.
\end{align*}
The last line is coming from the following form of Young's inequality: 
\begin{align*}
    ab\leq \varepsilon a^p + \frac{(p\varepsilon)^{1-q}}{q}b^q \quad \forall \ (a,b,\varepsilon)\in (0,\infty)^3, \quad \forall\ p,q\in (1,\infty),\ \frac1p+\frac1q=1.
\end{align*}
Hence, we can absorb the first term on the right-hand site by the left-hand site of \eqref{eq2-1}. Further, the definitions of $\rho$ and $\bu$ in \eqref{extended}, together with $\mathbf q_0=\rho(0)\bu(0)$, yield
%\begin{align*}
%    \int_\Omega \frac12 \rho |\bu|^2 \dd x &=\int_{\mf(t)} \frac12 \rho |\bu|^2 \dd x+\int_{\ms(t)} \frac12 \rho_s |\dot{\mathbf G} + \omega\times (x-\mathbf G)|^2 \dd x\\
%    &=\int_{\mf(t)} \frac12 \rho |\bu|^2 \dd x+\frac12 \rho_s\int_{\ms(t)} |\dot{\mathbf G}|^2 +2\dot{\mathbf G}\cdot (\omega\times (x-\mathbf G))+|\omega\times (x-\mathbf G)|^2 \dd x\\
%    &=\int_{\mf(t)} \frac12 \rho |\bu|^2 \dd x + \frac12 m |\dot{\mathbf G}|^2 + \frac12 \mathbb J \omega\cdot\omega + \int_{\ms(t)} \rho_s \dot{\mathbf G}\cdot (\omega\times (x-\mathbf G)) \dd x.
%\end{align*}
\begin{align*}
    \int_{\ms(0)} \frac{|\mathbf q_0|^2}{2\rho_0} \dd x &=\int_{\ms(0)} \frac12 \rho_s |\mathbf V_0 %+ \omega_0\times (x-\mathbf G_0)
    |^2 \dd x
    =\frac12 \rho_s\int_{\ms(0)} |\mathbf V_0|^2 %+2\mathbf V_0\cdot (\omega_0\times (x-\mathbf G_0))+|\omega_0\times (x-\mathbf G_0)|^2
    \dd x
    =\frac12 m |\mathbf V_0|^2. %+ \frac12 \mathbb J_0 \omega_0\cdot\omega_0 + \int_{\ms(0)} \rho_s \mathbf V_0\cdot (\omega_0\times (x-\mathbf G_0)) \dd x.
\end{align*}
Together with \eqref{eq2-1}, %and the fact that $\mathbf G_0= \frac{1}{m}\int_{\ms(0)} \rho_s x \dd x$,
we conclude the desired estimate \eqref{eq2}.
\end{proof}

\begin{Remark}
In the case of gravitational force, we have $\mathbf f=-g\mathbf e_3=-g\nabla [x\mapsto x_3]$, where $g>0$ is the acceleration due to gravity. Thus, $F=-gx_3$ and the energy inequality \eqref{eq2} becomes
\begin{align}\label{g:energy}
\begin{split}
&\sup_{t\in (0,T)} \int_{\mf(t)} \bigg(\frac{1}{2}{\rho}|\bu|^2+\frac{\rho^\gamma}{2(\gamma-1)}\bigg) \dd x + \sup_{t\in (0,T)} \int_{\ms(t)} \frac12 \rho |\bu|^2 \dd x \\
&\quad + \int_0^T\int_{\mf(t)} \big( 2\mu|\bD(\bu)|^2+\lambda|\div\bu|^2 \big) \dd x \dd t\\
&\leq \int_{\mf(0)}\bigg(\frac{1}{2}\frac{|\mathbf q_0|^2}{{\rho}_0} + \frac{\rho_0^\gamma}{\gamma-1}\bigg) \dd x+\frac{m}{2} |\mathbf V_0|^2 %+ \frac12 \mathbb J_0 \omega_0\cdot\omega_0
+ C(\gamma)\|g x_3\|_{L^\frac{\gamma}{\gamma-1}(\Omega)}^\frac{\gamma}{\gamma-1}\\
&\leq\int_{\mf(0)}\bigg(\frac{1}{2}\frac{|\mathbf q_0|^2}{{\rho}_0} + \frac{\rho_0^\gamma}{\gamma-1}\bigg) \dd x+\frac{m}{2} |\mathbf V_0|^2 %+ \frac12 \mathbb J_0 \omega_0\cdot\omega_0
+ C(\gamma) g^\frac{\gamma}{\gamma-1} (\mathrm{diam} \, \Omega)^{\frac{\gamma}{\gamma-1}+3},
\end{split}
%
%\begin{split}
%&\sup_{t\in (0,T)}\bigg[\int_{\mf(t)}\bigg(\frac{1}{2}{\rho}|\bu|^2+\frac{\rho^\gamma}{\gamma-1}\bigg) \dd x  +\frac{m}{2} |\dot{\mathbf G}(t)|^2+ \frac12 \mathbb J\omega\cdot\omega \bigg]
% + \int_0^T\int_{\mf(t)}\mu|\bD(\bu)|^2+(\mu+\lambda)|\div\bu|^2 \dd x \dd t\\
% &\leq C\bigg[\int_{\mf(0)}\bigg(\frac{1}{2}\frac{|\mathbf q_0|^2}{\rho_0} + \frac{\rho_0^\gamma}{\gamma-1}\bigg) \dd x + \frac{m}{2} |\mathbf V_0|^2 + \frac12 \mathbb J \omega_0\cdot\omega_0 + \|g x_3\|_{L^\frac{\gamma}{\gamma-1}(\Omega)}^\frac{\gamma}{\gamma-1}\bigg]\\
% &\leq C\bigg[\int_{\mf(0)}\bigg(\frac{1}{2}\frac{|\mathbf q_0|^2}{\rho_0} + \frac{\rho_0^\gamma}{\gamma-1}\bigg) \dd x + \frac{m}{2} |\mathbf V_0|^2 + \frac12 \mathbb J \omega_0\cdot\omega_0 + g^\frac{\gamma}{\gamma-1}\bigg].
% \end{split}
\end{align}
where we used that $x_3\leq \mathrm{diam} \, \Omega$ and $|\Omega|\leq (\mathrm{diam} \, \Omega)^3$.

If we define the initial energy by
\begin{align}\label{eq:E0}
E_0 = \int_{\mf(0)}\bigg(\frac{1}{2}\frac{|\mathbf q_0|^2}{\rho_0} +  \frac{\rho_0^\gamma}{\gamma-1}\bigg) \dd x + \frac{m}{2}|\mathbf V_0|^2, %+ \frac12 \mathbb J_0 \omega_0\cdot\omega_0,
\end{align}
we obtain from \eqref{g:energy} the following energy estimate:
\begin{align}\label{re:energy1}
\begin{split}
&\sup_{t\in (0,T)}\int_{\mf(t)}\bigg(\frac{1}{2}{\rho}|\bu|^2+\frac{\rho^\gamma}{2(\gamma-1)}\bigg)\dd x + \sup_{t\in (0,T)}\int_{\ms(t)} \frac{1}{2}{\rho}|\bu|^2 \dd x \\
& \quad +\int_0^T\int_{\Omega} \big( 2\mu|\bD(\bu)|^2+\lambda|\div\bu|^2 \big)\dd x \dd t\\
&\leq E_0+L(g,\gamma,\Omega),
\end{split}
\end{align}
where
\begin{align}\label{eq:L}
    L(g,\gamma,\Omega)=C(\gamma) g^\frac{\gamma}{\gamma-1} (\mathrm{diam} \, \Omega)^{\frac{\gamma}{\gamma-1}+3}.
\end{align}
In particular, this together with Korn's and Poincar\'e's inequality yields
\begin{align}\label{re:energy}
    \|\rho |\bu|^2\|_{L^\infty(0,T;L^1(\Omega))} + \|\rho\|_{L^\infty(0,T;L^\gamma(\mf(\cdot))}^\gamma + \|\bu\|_{L^2(0,T;W_0^{1,2}(\Omega))}^2 \leq C(\Omega,\mu,\gamma)(E_0+L).
\end{align}
%Note carefully that we do not have such an estimate for $\rho|_\ms=\rho_s$ {\cre in terms of $E_0$ and $L$}. Therefore, we have to distinguish between the fluid and solid part in the estimates {\cre given in} in Section \ref{sec:PfThm}.
\end{Remark}

%%%%%%%%%%%%%%%%%%%%%%%%%%%%%

%%%%%%%%%%%%%%%%%%%%%%%%%%%%%

\section{Construction of the test function}\label{sec3}
Assume that  $(\rho, \bu, {\mathbf G})$ is a weak solution of \eqref{eq1} satisfying the assumptions \eqref{a1}--\eqref{a6} in the time interval $(0,T_*)$ before collision. From now on we denote $\ms_h=\ms_h(t)=\ms(0)+(h(t)-h(0)){\mathbf e}_3$ and $\mf_{h}=\mf_h(t)=\Omega\setminus \overline{\ms_h(t)}$.

Collision can occur if and only if $\lim_{t\to T_*}h(t)=0$. Note further that $\mbox{\rm dist}(\ms_h(t),\partial \Omega)= \min\{h(t),d_0\}$ by assumptions \eqref{a2} and \eqref{a6}.

%%%%%%%%%%%%%%%%%%%%%%%%%%%%%

%%%%%%%%%%%%%%%%%%%%%%%%%%%%%

\subsection{Test function}
In this section, we will use the notation $a\lesssim b$ whenever there is a constant $C>0$ which is independent of $a$, $b$, $h$, and $T$ such that $a\leq C\, b$. Further, we will make use of cylindrical coordinates $(r,\theta,x_3)$ with the standard basis $(\mathbf e_r, \mathbf e_\theta, \mathbf e_3)$.
As in the papers \cite{VaretHillairet2010, VaretHill2012, VHW2015}, we construct a test function ${\mathbf w}_h$ associated with the solid particle $\ms_h$ frozen at distance $h$. This function will be defined for $h\in (0,h_M)$ with $h_M=\sup_{t\in [0,T_*)}h(t)$. Note that when $h\to 0$, a cusp arises in $\mf_h$, which is contained in% the domain
\begin{align}\label{omegah}
    \Omega_{h,r_0}=\{ x \in \mf_h: 0\leq r<r_0,\ 0\leq  x_3\leq  h+r^{1+\alpha},\ r^2=x_1^2+x_2^2\}.
\end{align}
For the sequel, we fix $h$ as a (small enough) positive constant and define $\psi(r):= h+r^{1+\alpha}$. Note that the common boundary $\del \Omega_{h,r_0}\cap \del \ms_h$ is precisely given by the set $\{0\leq r\leq r_0,\ x_3=\psi(r)\}$.

We use a similar function as in \cite{VaretHillairet2010}: define smooth functions $\chi, \eta$ satisfying
\begin{align}
&\chi=1 \text{ on } (-r_0,r_0)^2 \times (0,r_0), && \chi=0 \text{ on } \Omega\setminus \big( (-2r_0, 2r_0)^2 \times (0, 2r_0) \big)\label{chi}\\
&\eta=1 \text{ on } \mathcal{N}_{d_0/2}, && \eta=0 \text{ on } \Omega \setminus \mathcal{N}_{d_0}, %\text{ near } (-r_0, r_0)^2 \times \{ 0 \}.
\end{align}
where $d_0>0$ is as in \eqref{a6}, and $\mathcal{N}_\delta$ is a $\delta$-neighborhood of $\ms(0)$. Let us further set $\Phi(t)=t^2(3-2t)$,
\begin{align}\label{phih}
\phi_h(r,x_3)= \begin{cases}
    r/2 & \text{on } \ms_h,\\
    \frac{r}{2} (1-\chi(r,x_3))\eta(r,x_3-h+h(0)) + \chi(r,x_3)\frac{r}{2}\Phi\left(\frac{x_3}{\psi(r)}\right) & \text{on } \Omega \setminus \ms_h,
\end{cases}
\end{align}
and $\mathbf w_h = \nabla \times (\phi_h \mathbf e_\theta)$.
Note that we can write $\mathbf w_h$ as
\begin{align}\label{def:wh}
\mathbf w_h=-\del_3\phi_h \mathbf e_r+\frac1r \del_r(r\phi_h)\mathbf e_3.
\end{align}
Observe that the function $\mathbf w_h$ satisfies
\begin{align*}
\mathbf w_h|_{\partial \ms_h}=\mathbf e_3,\quad \mathbf w_h|_{\del \Omega}=0,\quad \div\mathbf w_h=0.
\end{align*}
Indeed, the divergence-free condition is obvious from the definition of $\mathbf w_h$. Further, since $\phi_h = r/2$ on $\ms_h$, we have $\mathbf w_h = \mathbf e_3$ there.
%\begin{align*}
%    \mathbf w_h|_{\del \ms_h} = \frac1r \del_r \bigg(\frac{r^2}{2} \Phi\bigg)(1)\mathbf e_3=\Phi(1)\mathbf e_3 + \frac{r}{2}\Phi'(1) \frac{\del_r\psi}{\psi}=\mathbf e_3.
%\end{align*}
Moreover, by definition of $\eta$ and $\chi$, we have $\phi_h=0$ on $\del \Omega \setminus \big( (-2 r_0, 2 r_0) \times \{ 0 \} \big)$ as long as $r_0$ and $h$ are so small that $h+r_0^{1+\alpha} \leq d_0 < r_0$. Lastly, in the annulus $\big( (-2r_0, 2r_0)^2 \setminus (-r_0, r_0) \big) \times \{ 0 \}$, we use $\Phi(0)=0$ and $\eta(r,h(0))=0$ for $r>\mathfrak{r}_0$ for some $\mathfrak{r}_0\in (d_0, r_0)$ to finally conclude $\mathbf w_h|_{\del\Omega}=0$, provided $h$ is sufficiently close to zero.

We summarize further properties of $\mathbf w_h$ in the following Lemma:
\begin{Lemma}\label{BdsWh}
$\mathbf w_h\in C_c^\infty(\Omega)$ and
\begin{align}\label{est1:wh}
\|\del_h\mathbf w_h\|_{L^\infty(\Omega\setminus \Omega_{h,r_0})} + \|\mathbf w_h\|_{W^{1,\infty}(\Omega\setminus \Omega_{h,r_0})}\lesssim 1.
\end{align}
Moreover,
\begin{align*}
\|\mathbf w_h\|_{L^p(\Omega_{h,r_0})} &\lesssim 1 \text{ for any } p<1+\frac3\alpha,\\
\|\del_h \mathbf w_h\|_{L^p(\Omega_{h,r_0})} + \|\nabla \mathbf w_h\|_{L^p(\Omega_{h,r_0})} &\lesssim 1 \text{ for any } p<\frac{3+\alpha}{1+2\alpha}.
\end{align*}
\end{Lemma}
\begin{proof}
We know from the definition of $\mathbf w_h$ in \eqref{def:wh} that $\mathbf w_h\in C_c^\infty(\Omega)$. Moreover, $\mathbf w_h$ is bounded outside a bounded region, so the first inequality \eqref{est1:wh} is obvious.

Due to the property \eqref{chi} of $\chi$, the function $\phi_h$ (see \eqref{phih}) in $\Omega_{h,r_0}$ (see \eqref{omegah}) becomes
\begin{equation*}
\phi_h(r,x_3)=\frac{r}{2}\Phi\bigg(\frac{x_3}{\psi(r)}\bigg)\quad \mbox{in}\quad \Omega_{h,r_0}.
\end{equation*}
By definition \eqref{def:wh} of $\mathbf w_h$, we have
\begin{align*}
\mathbf w_h = -\frac{r}{2}\Phi' \bigg(\frac{x_3}{\psi}\bigg)\frac1\psi \mathbf e_r + \Phi\bigg(\frac{x_3}{\psi}\bigg)\mathbf e_3 - \frac{r}{2}\Phi' \bigg(\frac{x_3}{\psi}\bigg)\frac{x_3 \del_r\psi}{\psi^2}\mathbf e_3 \quad \mbox{in}\quad \Omega_{h,r_0}.
\end{align*}
Further, $x_3\leq \psi$ in $\Omega_{h,r_0}$. Hence,
\begin{align}\label{bdsPhi}
    |\Phi| + |\Phi'| + |\Phi''|\lesssim 1,
\end{align}
leading to
\begin{align*}
|\mathbf w_h|\lesssim 1+\frac{r}{\psi}(1+\del_r \psi).
\end{align*}
Due to legibility, we will not write the argument of $\Phi$ in the sequel. Similarly, we obtain
\begin{align*}
|\del_r \mathbf w_h| &\lesssim \Phi' \frac1\psi + \frac{r}{2\psi}\Phi''\frac{x_3\del_r \psi}{\psi^2} + \frac{r}{2}\Phi'\frac{\del_r\psi}{\psi^2}\\
&\quad + \Phi'\frac{x_3\del_r\psi}{\psi^2} + \frac{r}{2}\Phi''\cdot \bigg(\frac{x_3\del_r\psi}{\psi^2}\bigg)^2+\frac{r}{2}\Phi'\frac{x_3\del_r^2\psi}{\psi^2} + r\Phi'\frac{x_3(\del_r\psi)^2}{\psi^3},\\
|\del_3\mathbf w_h| &\lesssim \frac{r}{2}\Phi'' \frac{1}{\psi^2} + \Phi'\frac1\psi + \frac{r}{2}\Phi''\frac{x_3\del_r\psi}{\psi^3}+\frac{r}{2}\Phi'\frac{\del_r\psi}{\psi^2},\\
|\del_h \mathbf w_h| &\lesssim \frac{r}{2}\Phi''\frac{x_3}{\psi^3}+\frac{r}{2}\Phi'\frac{1}{\psi^2}+\Phi'\frac{x_3}{\psi^2}+\frac{r \del_r\psi}{2}\Phi''\cdot \bigg(\frac{x_3}{\psi^2}\bigg)^2+\Phi'\frac{x_3 r\del_r\psi}{\psi^3}.
\end{align*}
Using again $x_3\leq \psi$ and the bounds \eqref{bdsPhi}, we have
\begin{align*}
    |\nabla\mathbf w_h| &\lesssim |\del_r\mathbf w_h|+|\del_3\mathbf w_h|+\bigg|\frac{\mathbf w_h\cdot\mathbf e_r}{r}\bigg| \lesssim \frac1\psi + \frac{r}{\psi^2} +\frac{r\del_r\psi}{\psi^2} + \frac{\del_r\psi}{\psi} + \frac{r(\del_r\psi)^2}{\psi^2} + \frac{r\del_r^2\psi}{\psi},\\
    |\del_h \mathbf w_h| &\lesssim \frac1\psi+\frac{r}{\psi^2}+\frac{r\del_r\psi}{\psi^2}.
\end{align*}
Note that these bounds hold independently of the specific form of $\psi$. In our setting, $\psi(r)=h+r^{1+\alpha}$. Thus, the proof of the remaining estimates on $\mathbf w_h$, $\nabla\mathbf w_h$ and $\del_h\mathbf w_h$ are based on the following result, which can be proven analogously to \cite[Lemma 13]{Hillairet2007}:
%\begin{Lemma}\label{Lem:est}
we have
\begin{align}\label{Lem:est}
    \int_0^{r_0} \frac{r^p}{(h+r^{1+\alpha})^q} \dd r\lesssim 1 \quad \forall \ (\alpha,p,q)\in (0,\infty)^3 \quad \mbox{  satisfying  }\quad p+1>q(1+\alpha).
\end{align}
%for all $(\alpha,p,q)\in (0,\infty)^3$ satisfying $p+1>q(1+\alpha)$.
%\end{Lemma}
%

Using the estimate \eqref{Lem:est}, we get
\begin{align*}
&\int_{\Omega_{h,r_0}} |\mathbf w_h|^p \dd x \lesssim 1+ \int_0^{r_0} \int_0^\psi \frac{r^{p+1}}{\psi^p}+\frac{r^{(1+\alpha)p+1}}{\psi^p} \dd x_3 \dd r \\
&\lesssim 1+ \int_0^{r_0} \frac{r^{p+1}}{\psi^{p-1}}+\frac{r^{(1+\alpha)p+1}}{\psi^{p-1}} \dd r \lesssim 1\\
\Leftrightarrow &\ p+2>(p-1)(1+\alpha) \ \mbox{ and } \ p(1+\alpha)+2>(p-1)(1+\alpha)\\
\Leftrightarrow &\ \alpha (p-1)<3.
\end{align*}
Using the estimates $r\del_r\psi\lesssim \psi$ and $r\del_r^2\psi\lesssim \del_r\psi$, we have
\begin{align*}
|\nabla \mathbf w_h| \lesssim |\del_r\mathbf w_h|+|\del_3\mathbf w_h|+\Big|\frac{\mathbf w_h \cdot \mathbf e_r}{r}\Big|\lesssim \frac1\psi+\frac{r}{\psi^2}+\frac{\del_r \psi}{\psi}, \qquad |\del_h \mathbf w_h| \lesssim \frac1\psi+\frac{r}{\psi^2}.
\end{align*}
In particular, it is enough to estimate $\nabla\mathbf w_h$, since the most restrictive term is $r/\psi^2$. Hence, we obtain
\begin{align*}
\int_{\Omega_{h,r_0}} |\nabla\mathbf w_h|^p \dd x &\lesssim \int_0^{r_0} \int_0^\psi \frac{r}{\psi^p} + \frac{r^{p+1}}{\psi^{2p}} + \frac{r (\del_r\psi)^p}{\psi^p} \dd x_3 \dd r \lesssim \int_0^{r_0} \frac{r}{\psi^{p-1}} + \frac{r^{p+1}}{\psi^{2p-1}} + \frac{r^{\alpha p+1}}{\psi^{p-1}} \dd r \lesssim 1\\
\Leftrightarrow &\ 2>(p-1)(1+\alpha) \ \mbox{ and } \ p+2>(2p-1)(1+\alpha) \ \mbox{ and } \ \alpha p+2>(p-1)(1+\alpha)\\
\Leftrightarrow &\ p<\frac{3+\alpha}{1+2\alpha}.
\end{align*}
\end{proof}
\begin{Remark}
Let us remark that the specific form of the function $\Phi$ is due to the following observation (see \cite[Section 3.1]{VaretHill2012}): searching for a minimizer of the energy functional $\int_{\mf_h} |\nabla \bu|^2 \dd x$ in the class
\begin{align*}
    \{\bu \in W_{\mathrm{loc}}^{1,2}(\mf_h): \bu=\nabla\times (\phi\mathbf e_\theta)=-\del_3\phi \mathbf e_r+\frac1r \del_r(r\phi)\mathbf e_3\},
\end{align*}
where $\phi$ satisfies the boundary conditions
\begin{align*}
    \del_3\phi |_{\del\ms_h}=\del_3\phi |_{\del\Omega}=0,\quad \del_r(r\phi) |_{\del\ms_h}=r, \quad \phi |_{\del\Omega}=0,
\end{align*}
and anticipating that most of the energy comes from the $x_3$-derivative as $h\to 0$, one ends up with a relaxed problem of searching a minimizer to
\begin{align*}
    \mathcal{E}_h=\int_{\{0<r<r_0,\ 0<x_3<\psi(r)\}} |\del_3 \bu_r|^2 \dd x = \int_{\{0<r<r_0,\ 0<x_3<\psi(r)\}} |\del_3^2 \phi|^2 \dd x
\end{align*}
in the class
\begin{align*}
    \{\bu \in W^{1,2}(\{0<r<r_0,\ 0<x_3<\psi(r)\}): \bu=-\del_3\phi \mathbf e_r+\frac1r \del_r(r\phi)\mathbf e_3\},
\end{align*}
supplemented with boundary conditions
\begin{align*}
    \del_3 \phi(r,\psi(r))=0, \quad \del_3 \phi(r,0)=0, \quad \phi(r,\psi(r))=\frac{r}{2}, \quad \phi(r,0)=0.
\end{align*}
According to the Euler--Lagrange equation $\del_3^4 \phi=0$, one easily finds that the unique minimizer of $\mathcal E_h$ is given by
\begin{align*}
    \phi_{\min}(r,x_3) = \frac{r}{2}\Phi\bigg(\frac{x_3}{\psi(r)}\bigg),\quad \Phi(t)=t^2(3-2t),\quad t\in [0,1].
\end{align*}
\end{Remark}
%%%%%%%%%%%%%%%%%%%%%%%%%%%%%

%%%%%%%%%%%%%%%%%%%%%%%%%%%%%

\subsection{Estimates near the collision -- Proof of Theorem \ref{theo1}}\label{sec:PfThm}

 In order to prove Theorem~\ref{theo1}, let $T_* \in (0, \infty]$ be the maximal existence time of the solution $(\rho, \bu)$. We will show that for any $T < T_*$, we have a uniform upper bound $T < \overline{T}(\gamma, \alpha, \mu, g, \mathbf w_h, \Omega, m, E_0, L)$, showing that the solution $(\rho, \bu)$ has a finite time interval of existence (that is, $T_* < \infty$) and hence collision occurs.

\begin{proof}[Proof of Theorem~\ref{theo1}]
Let $0<T<T_*$ and let $\zeta\in C^1_c([0,T))$ with $0\leq \zeta\leq 1$, $\zeta'\leq 0$, and $\zeta=1$ near $t=0$. (For instance, the properly extended function $\zeta_k(t)=\zeta(kt-(k-1)T)$ for some $k\geq 1$ and $\zeta(t)=\exp[T^{-2}-(T^2-t^2)^{-1}]$ will do.) We take $\zeta(t){\mathbf w}_{h(t)}$ as test function in the weak formulation of the momentum equation \eqref{momentum_weak} with the source term ${\mathbf f}=-g{\mathbf e}_3$, $g>0$. Recalling $\div\mathbf w_h=0$ and $\del_t \mathbf w_{h(t)}=\dot{h}(t) \del_h\mathbf w_{h(t)}$, we have the identity
\begin{align} \label{e1}
\begin{split}
&\int_0^T\zeta\int_\Omega \rho \bu\otimes \bu:\bD({\mathbf w}_h) \dd x \dd t + \int^T_0\zeta'\int_\Omega \rho \bu\cdot {\mathbf w}_h \dd x \dd t\\
&\quad + \int^T_0\zeta \dot{h}\int_\Omega \rho \bu\cdot {\del_h \mathbf w}_h \dd x \dd t - \int^T_0\zeta\int_\Omega {\bS}(\bu) :\bD({ \mathbf w}_h) \dd x \dd t\\
= &\int_0^T\zeta \int_\Omega \rho g \mathbf e_3\cdot\mathbf w_h \dd x \dd t - \int_\Omega \mathbf q_0\cdot \mathbf w_h \dd x\\
= &\int_0^T\zeta \int_{\ms_h} \rho g \mathbf e_3\cdot\mathbf w_h \dd x \dd t + \int_0^T\zeta \int_{\mf_h} \rho g \mathbf e_3\cdot\mathbf w_h \dd x \dd t - \int_\Omega \mathbf q_0\cdot \mathbf w_h \dd x.
\end{split}
\end{align}

Observe that we have $\mathbf w_h=\mathbf e_3$ on $\ms_h$, so for a sequence $\zeta_k\to 1$ in $L^1([0,T))$,
\begin{align*}
    &\int^T_0\zeta_k \int_{\ms_h} \rho g{\mathbf e}_3\cdot {\mathbf w}_h \dd x \dd t = \int_0^T\zeta_k \int_{\ms_h} \rho_s g \to mgT.
\end{align*}

In particular, for a proper choice of $\zeta$, it follows that
\begin{multline}\label{mom}
\frac12 mgT \leq \int_0^T\zeta\int_\Omega \rho \bu\otimes \bu:\bD({\mathbf w}_h) \dd x \dd t + \int^T_0\zeta'\int_\Omega \rho \bu\cdot {\mathbf w}_h \dd x \dd t + \int^T_0\zeta \dot{h}\int_\Omega \rho \bu\cdot {\del_h \mathbf w}_h \dd x \dd t\\
\quad - \int^T_0\zeta\int_\Omega {\bS}(\bu) :\bD({ \mathbf w}_h) \dd x \dd t
- \int_0^T\zeta \int_{\mf_h} \rho g \mathbf e_3\cdot\mathbf w_h \dd x \dd t + \int_\Omega \mathbf q_0\cdot \mathbf w_h \dd x
 = \sum_{j=1}^6 I_j.
\end{multline}

We will estimate each $I_j$ separately, and set our focus on the explicit dependence on $T$ and $m$. For the latter purpose, we split each density dependent integral into its fluid and solid part $I_j^f$ and $I_j^s$, respectively. To get a lean notation, we will use the symbol $a\lesssim b$ whenever there is a constant $C>0$ such that $a\leq C\, b$, where $C$ does not depend on $a$, $b$, $E_0$, $L$, and $T$.\\

$\bullet$ For $I_2^f$, we have by $\zeta'\leq 0$, $\zeta(T)=0$, and $\zeta(0)= 1$
\begin{align*}
    |I_2^f| &\leq -\int_0^T\zeta'\int_{\mf_h} \rho |\bu| |\mathbf w_h| \dd x \dd t = -\int_0^T\zeta'\int_{\mf_h} \sqrt{\rho} \sqrt{\rho} |\bu| |\mathbf w_h| \dd x \dd t\\
    &\leq -\int_0^T\zeta' \|\sqrt{\rho}\|_{L^{2\gamma}({\mf_h})} \|\sqrt{\rho}\bu\|_{L^2({\mf_h})} \|\mathbf w_h\|_{L^\frac{2\gamma}{\gamma-1}({\mf_h})} \dd t\\
    &\leq \|\rho\|_{L^\infty(0,T;L^\gamma({\mf(\cdot)}))}^\frac12 \|\rho |\bu|^2\|_{L^\infty(0,T;L^1(\Omega))}^\frac12 \|\mathbf w_h\|_{L^\infty(0,T;L^\frac{2\gamma}{\gamma-1}(\mf(\cdot)))} \zeta(0)\lesssim (E_0+L)^{\frac{1}{2\gamma}+\frac12},
\end{align*}
where we have used the estimate \eqref{re:energy} and Lemma \ref{BdsWh} under the condition 
\begin{align*}
    \frac{2\gamma}{\gamma-1}<1+\frac3\alpha \Leftrightarrow \alpha < \frac{3\gamma-3}{\gamma+1}.
\end{align*}

$\bullet$ For $I_2^s$, notice that $\mathbf w_h|_{\ms_h}=\mathbf e_3$, $\rho|_{\ms_h}=\rho_s$, and $\bu|_{\ms_h}=\dot{h}\mathbf e_3$. %+ \omega(t)\times (x-h\mathbf e_3)$.
Further, as before,% $\frac1m \int_{\ms_h} \rho_s x \dd x = \mathbf G_h = h\mathbf e_3 = \frac1m \int_{\ms_h} \rho_s h \mathbf e_3 \dd x$, and
\begin{align*}
    \int_{S_h} \frac12 \rho |\bu|^2 \dd x = \int_{S_h} \frac12 \rho_s |\dot{h}\mathbf e_3|^2 %+ |\omega\times(x-h\mathbf e_3)|^2\big)
    \dd x = \frac12 m |\dot{h}|^2. %+ \frac12 \mathbb J \omega\cdot\omega.
\end{align*}
Hence, we infer from the energy inequality \eqref{g:energy} that
\begin{align}\label{est:h}
    \sup_{t\in (0,T)} |\dot{h}| = (\sup_{t\in (0,T)} |\dot{h}|^2)^\frac12 \leq \sqrt{\frac2m}(E_0+L)^\frac12.
\end{align}
Thus,
\begin{align*}
    |I_2^s| = \bigg|\int_0^T \zeta' \int_{\ms_h} \rho_s \mathbf e_3\cdot \dot{h} \mathbf e_3 %+ \omega\times (x- h\mathbf e_3))
    \dd x \dd t\bigg| = \bigg|\int_0^T \zeta' \dot{h} m \dd t\bigg|\leq m \sup_{t\in (0,T)} |\dot{h}| \lesssim \sqrt{m}(E_0 + L)^\frac12.
\end{align*}

$\bullet$ For $I_3$, observe that $I_3^s=0$ due to $\del_h \mathbf w_h|_{\ms_h} = \del_h \mathbf e_3=0$. Next, by Sobolev embedding and \eqref{re:energy},
\begin{align*}
    \|\bu\|_{L^2(0,T;L^6(\Omega))} \lesssim \|\bu\|_{L^2(0,T;W_0^{1,2}(\Omega))} \lesssim (E_0+L)^\frac12.
\end{align*}
Thus,
%\begin{align*}
%    |I_3|&=|I_3^f| \leq \int_0^T\zeta |\dot{h}(t)|\, \|\rho\|_{L^\infty(0,T;L^\gamma(\mf(\cdot)))}^\frac12 \|\rho |\bu|^2\|_{L^\infty(0,T;L^1(\Omega))}^\frac12 \|\del_h \mathbf w_h\|_{L^\frac{2\gamma}{\gamma-1}(\mf(\cdot))} \dd t\\
%    &\lesssim (E_0+L)^{\frac{1}{2\gamma}+\frac12} \|\dot{h}\|_{L^\infty(0,T)} \|\zeta\|_{L^1(0,T)}\lesssim \sqrt{\frac1m} (E_0+L)^{\frac{1}{2\gamma}+1} T,
%\end{align*}
%where we have used the estimates \eqref{re:energy}, \eqref{est:h}, and Lemma \ref{BdsWh} under the condition
%\begin{align*}
%    \frac{2\gamma}{\gamma-1} < \frac{3+\alpha}{1+2\alpha} \Leftrightarrow \alpha < \frac{\gamma-3}{3\gamma+1}.
%\end{align*}
\begin{align*}
    |I_3|&=|I_3^f| \leq \int_0^T\zeta |\dot{h}(t)|\, \|\rho\|_{L^\infty(0,T;L^\gamma(\mf(\cdot)))} \|\bu\|_{L^2(0,T;L^6(\Omega))} \|\del_h \mathbf w_h\|_{L^\frac{6\gamma}{5\gamma-6}(\mf(\cdot))} \dd t\\
    &\lesssim (E_0+L)^{\frac{1}{\gamma}+\frac12} \|\dot{h}\|_{L^\infty(0,T)} \|\zeta\|_{L^1(0,T)}\lesssim \sqrt{\frac1m} (E_0+L)^{\frac{1}{\gamma}+1} T,
\end{align*}
where we have used the estimates \eqref{re:energy}, \eqref{est:h}, and Lemma \ref{BdsWh} under the condition
\begin{align*}
    \frac{6\gamma}{5\gamma-6} < \frac{3+\alpha}{1+2\alpha} \Leftrightarrow \alpha < \frac{9(\gamma-2)}{7\gamma+6}.
\end{align*}

$\bullet$ Regarding $I_4$, by using the fact that $\div\mathbf w_h=0$, we have
\begin{align*}
    \bS(\bu):\bD(\mathbf w_h)&=2\mu \bD(\bu):\bD(\mathbf w_h)+\lambda \div\bu \ \mathbb I:\bD(\mathbf w_h)=2\mu \bD(\bu):\bD(\mathbf w_h)+\lambda \div\bu \div\mathbf w_h\\
    &=2\mu \bD(\bu):\bD(\mathbf w_h).
\end{align*}
Hence, using the bounds on $\bD(\bu)$ already obtained in \eqref{re:energy1}, we calculate
\begin{align*}
    |I_4| &\lesssim \int_0^T \zeta \|\bD(\bu)\|_{L^2(\Omega)} \|\nabla \mathbf w_h\|_{L^2(\Omega)} \dd t \leq \|\zeta\|_{L^2(0,T)} \|\bD(\bu)\|_{L^2((0,T)\times \Omega)} \|\nabla\mathbf w_h\|_{L^\infty(0,T;L^2(\Omega))}\\
    &\lesssim (E_0+L)^\frac12 T^\frac12,
\end{align*}
where we have used Lemma \ref{BdsWh} under the condition
\begin{align*}
    2<\frac{3+\alpha}{1+2\alpha} \Leftrightarrow \alpha< \frac13.
\end{align*}

$\bullet$ For $I_5=I_5^f$,
\begin{align*}
    |I_5| &\leq g \int_0^T\zeta \|\rho\|_{L^\gamma(\mf_h)}\|\mathbf w_h\|_{L^\frac{\gamma}{\gamma-1}(\Omega)} \leq g \|\zeta\|_{L^1(0,T)} \|\rho\|_{L^\infty(0,T;L^\gamma(\mf(\cdot)))} \|\mathbf w_h\|_{L^\infty(0,T;L^\frac{\gamma}{\gamma-1}(\Omega))}\\
    &\leq g (E_0+L)^\frac1\gamma T,
\end{align*}
by using Lemma \ref{BdsWh} under the condition
\begin{align*}
    \frac{\gamma}{\gamma-1}<1+\frac3\alpha \Leftrightarrow \alpha< 3-\frac3\gamma.
\end{align*}

$\bullet$ Similar to $I_2^f$, we have for $I_6^f$ the estimate
\begin{align*}
    |I_6^f| \leq \|\mathbf q_0\|_{L^\frac{2\gamma}{\gamma+1}(\mf(0))}\|\mathbf w_h\|_{L^\infty(0,T;L^\frac{2\gamma}{\gamma-1}(\Omega))} \lesssim \bigg\|\frac{|\mathbf q_0|^2}{\rho_0}\bigg\|_{L^1(\mf(0))}^\frac12 \|\rho_0\|_{L^\gamma(\mf(0))}^\frac12 \lesssim (E_0+L)^{\frac12+\frac{1}{2\gamma}}.
\end{align*}

$\bullet$ For $I_6^s$, where $\mathbf w_h=\mathbf e_3$ and $\mathbf q_0=\rho(0)\bu(0)=\rho_s \dot{h} \mathbf e_3$, %+\omega\times (x-h\mathbf e_3))$,
we have similarly to $I_2^s$ that
\begin{align*}
    |I_6^s| = \bigg|\int_{\ms(0)} \mathbf q_0\cdot \mathbf e_3 \dd x\bigg| = \bigg|\int_{\ms(0)} \rho_s \dot{h} \dd x\bigg| \leq m \|\dot{h}\|_{L^\infty(0,T)} \lesssim \sqrt{m}(E_0+L)^\frac12.
\end{align*}

$\bullet$ Let us turn to $I_1$. Due to $\mathbf w_h|_{\ms_h}=\mathbf e_3$, we see that $I_1^s=0$ since $\bD(\mathbf w_h)=0$ there. Hence, we calculate
\begin{align*}
    |I_1|&= |I_1^f| \lesssim \int_0^T \zeta \|\rho\|_{L^\gamma(\mf_h))} \|\bu\|_{L^6(\Omega)}^2 \|\nabla\mathbf w_h\|_{L^\frac{3\gamma}{2\gamma-3}(\Omega)} \\
    &\lesssim \|\rho\|_{L^\infty(0,T;L^\gamma(\mf_h))} \|\nabla\mathbf w_h\|_{L^\infty(0,T;L^\frac{3\gamma}{2\gamma-3}(\Omega))} \int_0^T\zeta \|\nabla\bu \|_{L^2(\Omega)}^2\\
    &\lesssim (E_0+L)^\frac1\gamma \|\zeta \|_{L^\infty(0,T)}\|\nabla\bu\|_{L^2((0,T)\times\Omega)}^2 \lesssim (E_0+L)^{\frac1\gamma+1},
\end{align*}
by using the estimate \eqref{re:energy} and Lemma \ref{BdsWh} under the condition
\begin{align*}
    \frac{3\gamma}{2\gamma-3}<\frac{3+\alpha}{1+2\alpha} \Leftrightarrow \alpha<\frac{3(\gamma-3)}{4\gamma+3}.
\end{align*}

Note further that for any $\gamma\geq 3$,
\begin{align*}
    \frac{3(\gamma-3)}{4\gamma+3}\leq \min\bigg\{\frac{3\gamma-3}{\gamma+1}, \frac{9(\gamma-2)}{7\gamma+6}, 3-\frac{3}{\gamma}\bigg\},
\end{align*}
and that all estimates are independent of the choice of $\zeta$. Hence, we can take a sequence $\zeta_k\to 1$ in $L^\infty([0,T))$ without changing the bounds obtained. In turn, collecting all estimates above, we finally arise at
\begin{equation*}
    \frac12 mgT \leq C_0 (1+\sqrt{m}+\sqrt{m}^{-1}) \bigg((E_0+L)^{\frac12+\frac{1}{2\gamma}} + (E_0+L)^\frac12 + g(E_0+L)^\frac1\gamma + (E_0+L)^{1+\frac1\gamma} \bigg) (1 + T^\frac12 + T),
\end{equation*}
which after dividing by $\frac12 m$ and using Young's inequality on several terms, leads to
\begin{align}\label{infinalIneq}
    gT \leq C_0 (m^{-1}+m^{-\frac12}+m^{-\frac32}) \bigg(1 + (E_0+L)^{1+\frac{1}{\gamma}} + g(E_0+L)^\frac1\gamma \bigg) (1 + T),
\end{align}
where $C_0$ only depends on $\gamma, \alpha, \mu$, the bounds on $\mathbf w_h$ obtained in Lemma \ref{BdsWh}, and the Sobolev and Korn constant of $\Omega$, provided
\begin{align*}
    \gamma>3 \quad \text{and} \quad \alpha<\min \bigg\{ \frac13, \frac{3(\gamma-3)}{4\gamma+3} \bigg\}.
\end{align*}

Recalling the definitions of $E_0$ from \eqref{eq:E0} and $L$ from \eqref{eq:L} as
\begin{align*}
    E_0 &= \int_{\mf(0)}\bigg(\frac{1}{2}\frac{|\mathbf q_0|^2}{\rho_0} +  \frac{\rho_0^\gamma}{\gamma-1}\bigg) \dd x + \frac{m}{2}|\mathbf V_0|^2,\\ %+ \frac12 \mathbb J_0 \omega_0\cdot\omega_0,\\
    %\mathbb J_0 &= \int_{\ms(0)}\rho_s\Big(|x-\mathbf{G}_0|^2{\mathbb I}-(x-\mathbf{G}_0)\otimes (x-\mathbf{G}_0)\Big) \dd x,\\
    L &= C(\gamma)g^\frac{\gamma}{\gamma-1}(\rm diam\,  \Omega)^{\frac{\gamma}{\gamma-1}+3},
\end{align*}

we see that collision can occur only if the solid's mass in \eqref{infinalIneq} is large enough, meaning in fact it's density is very high. Since $E_0$ depends on solid's mass, %and, through $\mathbb J_0$, on solid's density,
we require the solid initially to have low vertical speed. %and rotational speed.
More precisely, choosing $\mathbf V_0$ %and $\omega_0$
such that $|\mathbf V_0|=\mathcal O(m^{-\frac12})$, and choosing $m$ high enough such that
\begin{align}\label{finalIneq}
    C_0 (m^{-1}+m^{-\frac12}+m^{-\frac32}) \bigg(1 + (E_0+L)^{1+\frac{1}{\gamma}} + g(E_0+L)^\frac1\gamma \bigg) < g,
\end{align}
the solid touches the boundary of $\Omega$ in finite time, ending the proof of Theorem \ref{theo1}.
\end{proof}
\subsection{Discussion on large $\alpha$}
To end this section, let us briefly explain what is the difficulty for larger values of $\alpha$. To this end, we recall how the argument for incompressible fluids in two dimensions works (see \cite[Section~4]{VaretHillairet2010}).\\

As can be seen from the above proof, one has to carefully estimate the term $\int_\mf \bS(\bu):\bD(\mathbf w_h) \dd x$. Indeed, one may find a pressure $q_h$ such that the couple $(\mathbf w_h, q_h)$ is a good approximation to the solution of the Stokes equations in $\mf$ (in a sense to be specified), and that for all $\bu \in W_0^{1,2}(\Omega)$ with $\div\bu=0$ and $\bu|_{\ms}=\dot{h}\mathbf e_3$
\begin{align*}
    \int_\mf |(\Delta\mathbf w_h-\nabla q_h)\cdot\bu |\dd x\lesssim \|\bu\|_{W_0^{1,2}(\Omega)}.
\end{align*}
For such solenoidal $\bu$, integration by parts then gives
\begin{align*}
    \int_\mf \bS(\bu):\bD(\mathbf w_h) \dd x &= -\int_\mf \bu \cdot \Delta \mathbf w_h \dd x + \int_{\del\mf} \bu \cdot \bD(\mathbf w_h) \cdot \mathbf n \dd S\\
    &= -\int_\mf \bu \cdot (\Delta \mathbf w_h - \nabla q_h) \dd x + \dot{h} \int_{\del\ms} \mathbf e_3 \cdot \bigg(\frac{\del \mathbf w_h}{\del\mathbf n}-q_h\mathbf n \bigg) \dd S.
\end{align*}
From the bounds on $\bu$, the first term is estimated by a constant, whereas the last term (the drag term) gives particularly the (no-)collision result. More precisely, one can show that the last integral (and, in fact, $\|\nabla\mathbf w_h\|_{L^2(\mf)}^2$ itself) behaves like $h^{-\beta}$ for $\beta=3\alpha/(1+\alpha)>0$. Estimating the remaining terms in the momentum equation, for $\alpha \geq \frac12$ leading to $\beta \geq 1$, one arrives at
\begin{align*}
    |\log h(T)| \sim \int_0^T \frac{\dot{h}(t)}{h(t)^\beta} \dd t = \mathcal{O}(T).
\end{align*}
This means that $h$ cannot vanish in finite time, thus the body stays away from $\del\Omega$. On the other hand, if $\alpha<\frac12$ such that $\beta<1$, a similar argument shows that $h$ vanishes in finite time provided the solid's density is greater than the fluid's one. In this context, it is worth noting that for the three-dimensional case, the authors in \cite{VaretHill2012} showed that collision can occur for \emph{any} value of $\alpha$, which seems to be a contrary result to the one of \cite{VaretHillairet2010} for dimension two.\\

The difficulty for the compressible case is now the following: anticipating that the collision/no-collision result should also stem from the term $\int_\mf \bS(\bu):\bD(\mathbf w_h) \dd x$, which does not contain the density and thus should behave like in the incompressible case, the ``best'' test function would still be $\mathbf w_h$ with the corresponding pressure $q_h$. Integration by parts then yields an additional term
\begin{align}\label{qh}
    \int_\mf q_h \div\bu \dd x,
\end{align}
which \textit{a priori} does not vanish since $\div \bu \neq 0$. Note further that no time derivative $\dot{h}$ is contained in this integral. It is not clear to us how to handle this integral, and whether or not it gives an additional contribution to the drag term and, in turn, to the (no-)collision.

%%%%%%%%%%%%%%%%%%%%%%%%%%%%%

%%%%%%%%%%%%%%%%%%%%%%%%%%%%%
\section{Results on No-collision }\label{sec4}
 In this section, we want to investigate a different setting, namely the case of a spherical solid with feedback control. We shall show that under a smallness condition, no collision occurs in this case. Let $\Omega \subset \mathbb{R}^3$ be a bounded domain with smooth boundary occupied by a fluid and a rigid body. We denote by $\mathcal{S}(t) \subset \Omega$, the domain of the rigid body of center $\mathbf{G}(t)$, where $t \in \mathbb{R}_{+}$ is the time variable.
We suppose that the fluid domain $\mf(t) = \Omega \setminus \overline{\mathcal{S}(t)}$ is connected.

The fluid is modeled by the compressible Navier-Stokes system whereas the motion of the rigid body is governed by the balance equations for linear and angular momentum. We also assume the no-slip boundary conditions. The equations of motion of fluid-structure are:
\begin{equation}\label{eq1feed}
\begin{cases}
\partial_t \rho_f+\div(\rho_f \bu_f)=0 & \mbox{ in }\mf(t),\\
\partial_t(\rho_f \bu_f)+\div(\rho_f \bu_f\otimes \bu_f)-\div{\bS(\bu_f)}+\nabla p=0 & \mbox{ in }\mf(t),\\
\bu_f =\dot{\mathbf G}(t)+ \omega(t)\times { (x-{\mathbf G}(t))} & \mbox{ on } {\partial \mathcal{S}(t)},\\
\bu_f=0 & \mbox{ on } {\partial \Omega},\\
m\ddot{\mathbf G}=- \int_{\partial \ms}\Big({\bS}(\bu_f)- p{\mathbb I}\Big){\mathbf n} \dd S + {\mathbf w},\\
\frac{\mathrm d}{\mathrm d t}(\mathbb J \omega)=-\int_{\partial \ms} (x-{\mathbf G}) \times \Big({\bS}(\bu_f)- p{\mathbb I}\Big){\mathbf n} \dd S,\\
\rho_f (0)=\rho_0,\ \rho_f\bu_f(0)=\mathbf q_0,\ \mathbf{G}(0)= \mathbf G_0,\ \dot{\mathbf G}(0)=\mathbf V_0,\ \omega(0)=\omega_0 & \mbox{ in }\mf(0).
\end{cases}
\end{equation}
Here $\mathbf{w}(t)$ is a feedback law of the form
\begin{equation}\label{feedback}
\mathbf{w}(t)= k_p(\mathbf{G}_1-\mathbf{G}(t))-k_d \dot{\mathbf{G}}(t).
\end{equation}
In control engineering, this feedback \eqref{feedback} is known as a proportional-derivative (PD) controller. The feedback $\mathbf{w}(t)$ is generated by a spring (with spring constant $k_p>0$) and a mechanical damper (with constant $k_d\geq 0$) connected between $\mathbf{G}(t)$ (centre of mass of $\mathcal{S}(t)$) and a fixed point $\mathbf{G}_1\in\Omega$. The definition of weak solutions of the system \eqref{eq1feed} is similar to Definition \ref{weaksol:def}, where the weak form of the momentum equation \eqref{momentum_weak} is replaced by 
\begin{multline}\label{momentum_weak1}
\int_0^t \int_\Omega \left[(\rho\bu)\cdot \partial_t \phi+ (\rho\bu \otimes \bu) : \bD(\phi) + p(\rho) \div \phi - \bS(\bu) : \bD(\phi) \right] \dd x \dd \tau\\
= \int_\Omega [\rho(t)\bu(t)\cdot \phi(t) - \mathbf q_0\cdot \phi(0)]\dd x + \int_0^T \mathbf{w}\cdot \ell_{\phi} \dd \tau,
 \end{multline}
for any $\phi \in C_c^\infty([0,T)\times \Omega)$ with $\phi(t,y)=\ell_\phi(t) + \omega_\phi(t)\times (y-\mathbf G(t))$ in a neighborhood of $\ms(t)$,
and the energy inequality \eqref{StdEI} is replaced by
\begin{align}\label{StdEI1}
\begin{split}
&\int_\Omega\left( \frac12 \rho(t,x)|\bu(t,x)|^2+\frac{\rho^\gamma(t,x)}{\gamma-1} \right) \dd x + \int_0^t\int_\Omega\left(2\mu |\bD(\bu)|^2+\lambda|\div\bu|^2\right)(\tau, x) \dd x \dd \tau\\
&\leq \int_{\Omega} \left( \frac12 \frac{|\mathbf q_0(x)|^2}{\rho_0(x)} + \frac{\rho_0^\gamma(x)}{\gamma-1}\right)\dd x + \int_0^t \mathbf{w}\cdot\dot{\mathbf{G}} \dd \tau,
\end{split}
\end{align}
where we extended $\rho,\mathbf u$ as in \eqref{extended}, and $\mathbf q_0=\rho(0) \bu(0)$.

The existence of weak solutions for system \eqref{eq1feed} can be established by following \cite[Theorem~4.1]{Feireisl2004}, and the existence of strong solutions can be found in \cite[Theorem 1.1]{roy2019stabilization}. We can use the energy estimate \eqref{StdEI1} for the system \eqref{eq1feed} with the feedback law \eqref{feedback} to obtain the following no-collision result:
\begin{Proposition}\label{energyest}
Let $\mathbf{G}_1 \in \Omega$ with $\operatorname{dist}(\mathbf{G}_1,\partial \Omega)>1$ and assume that $\mathbf{w}$ satisfies the feedback law \eqref{feedback}. Let  $(\rho,\mathbf{u},\mathbf{G})$ be a weak solution of the system \eqref{eq1feed} and the initial data satisfy \eqref{init}--\eqref{init2}. Then
\begin{align}\label{energy estimate full system}
\begin{split}
&\int_{\mf(t)} \left( \frac{\rho}{2}|\mathbf{u}|^2+\frac{a}{\gamma-1}\rho^{\gamma} \right) \dd x + \frac{m}{2}|\dot{\mathbf{G}}|^2 + \frac12 \mathbb{J} \omega\cdot\omega + \frac{k_p}{2}|\mathbf{G}_1-\mathbf{G}(t)|^2 + {k_d}\int_0^t |\dot{\mathbf{G}}(t)|^2 \dd \tau \\
&+ \int_{0}^{t}\int_{\mf(t)}\left(2\mu |\bD({\mathbf{u}})|^2 + \lambda|\div\mathbf{u}|^2\right) \dd x \dd \tau\\
&\leq C\left(\int_{\mf(0)}\left( \frac{1}{2}\frac{|\mathbf{q}_0|^2}{\rho_0}+\frac{a}{\gamma-1}\rho_{0}^{\gamma} \right) \dd x+\frac{m}{2}|\mathbf{V}_0|^2 + \frac12 \mathbb{J}\omega_0\cdot \omega_0 + \frac{k_p}{2}|\mathbf{G}_1-\mathbf{G}_0|^2\right).
\end{split}
\end{align}
Moreover, there exist $\delta, \varepsilon > 0$ such that if 
\begin{equation*}
\int_{\mf(0)}\left( \frac{|\mathbf q_0|^2}{2 \rho_0}+\frac{a}{\gamma-1}\rho_{0}^{\gamma} \right) \dd x + \frac{m}{2}|\mathbf{G}_0|^2 + \frac12 \mathbb{J}\omega_0 \cdot \omega_0 + \frac{k_p}{2}|\mathbf{G}_1-\mathbf{G}_0|^2 \leq \delta,
\end{equation*}
then \begin{equation}\label{noc}
\operatorname{dist}(\mathbf{G}(t),\partial\Omega) \geq 1+\varepsilon .
\end{equation}
\end{Proposition}
\begin{proof}
As $\mathbf{w}(t)= k_p(\mathbf{G}_1-\mathbf{G}(t))-k_d \dot{\mathbf{G}}(t)$, we obtain 
\begin{equation}\label{feeden}
-\mathbf{w}\cdot \dot{\mathbf G} = -k_p(\mathbf{G}_1-\mathbf{G}(t))\cdot \dot{\mathbf{G}}(t) + k_d|\dot{\mathbf G}(t)|^2 = \frac{\rm d}{{\rm d} t}\left(\frac{k_p}{2}|\mathbf{G}_1-\mathbf{G}(t)|^2\right)+k_d|\dot{\mathbf G}(t)|^2.
\end{equation}
We use the relation \eqref{feeden} in the energy inequality \eqref{StdEI1} and definition of extended $\rho,\mathbf u$ as in \eqref{extended} to obtain our required estimate \eqref{energy estimate full system}.

In order to establish \eqref{noc}, we can use the estimate \eqref{energy estimate full system} to obtain
\begin{multline}\label{energy estimate full system1}
 |\mathbf{G}_1-\mathbf{G}(t)|^2  
\leq \frac{2C}{k_p}\left(\int_{\mf(0)}\left( \frac{1}{2}\frac{|\mathbf{q}_0|^2}{\rho_0}+\frac{a}{\gamma-1}\rho_{0}^{\gamma} \right) \dd x+\frac{m}{2}|\mathbf{V}_0|^2 + \frac12 \mathbb{J}\omega_0\cdot \omega_0+\frac{k_p}{2}|\mathbf{G}_1-\mathbf{G}_0|^2\right).
\end{multline}
Finally, we use the fact $\operatorname{dist}(\mathbf{G}_1,\partial \Omega)>1$ and choose $\varepsilon>0$ such that 
\begin{equation*}
\frac{2C}{k_p}\left(\int_{\mf(0)}\left( \frac{1}{2}\frac{|\mathbf{q}_0|^2}{\rho_0}+\frac{a}{\gamma-1}\rho_{0}^{\gamma} \right) \dd x+\frac{m}{2}|\mathbf{V}_0|^2 + \frac12 \mathbb{J}\omega_0\cdot \omega_0+\frac{k_p}{2}|\mathbf{G}_1-\mathbf{G}_0|^2\right) < \varepsilon^2
\end{equation*}
to conclude \eqref{noc}.
%\begin{equation*}
%\operatorname{dist}(\mathbf{G}(t),\partial\Omega) \geq 1+\varepsilon .
%\end{equation*}
\end{proof}
\begin{Remark}
If we allow higher regularity, then we can obtain the no-collision result even without the external PD-controller. As in \cite[Theorem 1.2]{HMTT}, we can establish the following result for the smooth rigid body $\ms(t)$: let $2<p<\infty$, $3<q<\infty$ satisfy the condition $\frac{1}{p}+\frac{1}{2q}\neq \frac{1}{2}$. Let the initial data satisfy 
\begin{equation*}
\rho_0 \in W^{1,q}(\mf(0)),\quad \mathbf{u}_0\in B^{2(1-1/p)}_{q,p}(\mf(0)) \footnote{Let $k\in\mathbb{N}$. For every $0<s<k$, $1\leq p <\infty$, $1\leq q <\infty$, we define the Besov spaces $B^s_{q,p}(\Omega)$ by real interpolation of Sobolev spaces: $B^s_{q,p}(\Omega)=(L^q(\Omega),W^{k,q}(\Omega))_{s/k,p}$. Please see \cite{MR781540} for a detailed presentation of Besov spaces.}, \quad \min_{\overline{\mf(0)}} \rho_0 > 0,
\end{equation*}
\begin{equation*}
\mathbf{G}_0\in \mathbb{R}^3,\quad \mathbf{V}_0\in \mathbb{R}^3,\quad \omega_0\in \mathbb{R}^3,
\end{equation*}
\begin{equation*}
\frac{1}{|\mf(0)|}\int_{\mf(0)}\rho_0 = \overline{\rho}>0 ,
\end{equation*}
\begin{equation*}
\mathbf{u}_0=0 \mbox{ on }\partial\Omega,\quad \mathbf{u}_0=\mathbf{V}_0 + \omega_0\times (y-\mathbf{G}_0) \mbox{ on }\partial\mathcal{S}(0).
\end{equation*}
If there exists $\delta > 0$ and $\varepsilon >0$ such that
\begin{equation*}
\|(\rho_0-\overline{\rho},\mathbf{u}_0,\mathbf{V}_0, \omega_0)\|_{W^{1,q}\times B^{2(1-1/p)}_{q,p} \times \mathbb{R}^3 \times \mathbb{R}^3}\leq \delta,
\end{equation*}
\begin{equation*}
\operatorname{dist}(\ms(0),\partial\Omega)\geq \varepsilon >0,
\end{equation*}
then
\begin{equation*}
\operatorname{dist}(\ms(t),\partial\Omega)\geq \frac{\varepsilon}{2}\mbox{ for all }t\in [0,\infty).
\end{equation*}
\end{Remark}

\section*{Acknowledgment}
{\it \v S. N. and F. O. have been supported by the Czech Science Foundation (GA\v CR) project 22-01591S.  Moreover, \it \v S. N.  has been supported by  Praemium Academiæ of \v S. Ne\v casov\' a. The Institute of Mathematics, CAS is supported by RVO:67985840. This research is supported by the Basque Government through the BERC 2022-2025 program and by the Spanish State Research Agency through BCAM Severo Ochoa excellence accreditation CEX2021-01142-S and through Grant PID2023-146764NB-I00. A. R would like to thank the Alexander von Humboldt-Foundation and Grant RYC2022-036183-I.}
\bibliography{reference}
\bibliographystyle{siam}

\end{document}